\documentclass[a4paper]{article}

\usepackage{amsmath}
\usepackage{amssymb}
\usepackage{bbm}
\usepackage{url}
\usepackage{stackrel}
\usepackage{color}
\usepackage{float}
\usepackage{subfig}
\usepackage{multirow}
\usepackage{amsfonts}
\usepackage{graphicx}
\usepackage[left=2cm,right=2cm,top=2cm,bottom=2cm]{geometry}
\newtheorem{theorem}{Theorem}
\newtheorem{proposition}{Proposition}

\newtheorem{definition}[theorem]{Definition}
\newtheorem{algorithm}{Algorithm}

\begin{document}
%\title{A globally convergent direct search method using non-monotone line search for bound constraints optimization}
%\title{Line search strategies in pattern search methods for box constraints optimization}
\title{A pattern search bound constrained optimization method with a nonmonotone line search strategy}

\author{Johanna A. Frau and Elvio A. Pilotta \thanks{Facultad de Matem\'atica, Astronom\'{\i}a y F\'{\i}sica, Universidad Nacional de C\'ordoba, CIEM (CONICET), Medina Allende s/n, Ciudad Universitaria (5000) C\'ordoba, Argentina. E-mail: \{jfrau, pilotta\}@famaf.unc.edu.ar}}

\maketitle

\begin{abstract}
A new pattern search method for bound constrained optimization is introduced. The proposed algorithm employs the coordinate directions, in a suitable way, with a nonmonotone line search for accepting the new iterate, without using derivatives of the objective function. The main global convergence results are strongly based on the relationship between the step length and a stationarity measure.
Several numerical experiments are performed using a well known set of test problems. Other line search strategies were tested and compared with the new algorithm.

\


\

%\subclass{90C30 \and 65K05
%}
\end{abstract}
\hrulefill

\section{Introduction}\label{intro}

In this paper, we propose a new algorithm to solve bound constrained optimization problems where the derivatives of the objective function are not available. So, the problem of interest is

\begin{equation}\label{problem}
\begin{array}{lc }
\mbox{minimize} & f(x) \\
\mbox{subject to} & x \in \Omega 
\end{array}
\end{equation}
where $f: \mathbbm{R}^n \rightarrow \mathbbm{R}$  and $\Omega = \{x \in \mathbbm{R}^n \, | \, l \leq x \leq u\}$ with $-\infty \leq l < u \leq \infty$. We assume that the objective function is continuosly differentiable on $\Omega$ but the derivative information is unrealiable or non-existent. This situation frequently appears in many real world applications where functional values $f(x)$ require complex simulations or the function contains noise. This kind of situations may arise in applications from molecular geometry \cite{ANRV04,MM94}, medical image registration \cite{OB07}, shape and design optimization \cite{BDFST98,DV04,MWDM04}. There are many problems where the functional values come from practical experiments, so  the explicit formulation of the objective function is not available and Quasi-Newton or finite difference methods are not applicable. Derivative-free optimization has received considerable attention from the optimization community during the last years, including the establishment of solid mathematical foundations for many of the methods considered in practice. Particularly, pattern search methods have succeeded where more elaborate approaches fail. These methods belong to the family of direct search methods, characterized by unsophisticated implementations, the abscence of the construction of a model of the objective function and the use of pattern matrices to explore directions around the current iterate. See \cite{AD03,CSV09,HJ61,KLT03,T97}.

\

Pattern search methods were initially introduced by Hooke and Jeeves \cite{HJ61} for unconstrained optimization problems and lately analized and formally presented by Torczon et. al.  \cite{KLT03,T97}. More recently, some strategies were adapted from derivative-based methods  and incorporated to pattern search methods. For instance, in \cite{DMR08} the authors introduced a global strategy, based in the ideas developed in {\cite{GLL86,LMR06,LF00,LS02}, that uses a nonmonotone line search scheme in a pattern search algorithm. 

\

In \cite{LT99} Lewis and Torczon extended the pattern search method for the bound constrained case although they did not perform numerical experiments. This problem was also studied in \cite{AEP11} using polynomial interpolation and trust region strategies, which is a quite different approach to pattern search methods.

\

In this article, we propose a pattern search method that includes a nonmonotone line search as globalization strategy for a bound constrained optimization problem (\ref{problem}). The new algorithm is based on the ideas introduced in \cite{DMR08} for the unconstrained case, however the proofs of the main convergence results of our method use a completely different philosophy . To prove the global convergence, we use the stationarity measure $\chi(x)$ defined in \cite{KLT03}. This measure takes into account the degree to which the directions of the steepest descent point outward  with respect to the portion of the feasible region near $x$. In order to validate our algorithm, we perform several numerical experiments and comparisons with other well established algorithm. After that, we extend our benchmark study by incorporating other line search strategies \cite{A66,C09,KLNS15,NS13,ZH04}. 

\

This paper is organized as follows: some definitions and preliminary results are given in Section 2. The new algorithm is introduced in Section 3. Convergence results are stated in Section 4. Numerical experiments are presented and analized in Section 5. Finally, some conclusions are given in Section 6.

\subsection*{Notation}
In this work, $e^{(i)}$ denotes the i-th canonical vector in $\mathbb{R}^{n}$ and  
$\|\cdot\|$ will be the Euclidean norm. Also, $int(\Omega)$ denotes the largest open set contained in $\Omega$.

\

\section{Definitions and preliminary results}

In this section, we present some definitions and results which are necessary in order to guarantee  convergence of the method that we propose to solve problem (\ref{problem}). 

\

The following two definitions are the basis of the theory of convergence and they are widely used in the context of optimization. The first definition refers to the cone $K$ generated by the set of all nonnegative linear combinations of vectors of a given set. The second definition includes those vectors that make an angle of 90$^\circ$ or more with each element of $K$, the polar cone ${K^\circ}$. 

\begin{definition}
Let $\textit{D}= \lbrace v^{(1)}, v^{(2)}, \ldots, v^{(r)} \rbrace$ be a set of $\textit{r}$ vectors in $\mathbb{R}^{n}$. The set $\textit{D}$ generates the cone $K$ if 
$$
K=\lbrace u | u= \sum_{i=1}^r c^{(i)} v^{(i)}, c^{(i)} \geq 0, \mbox{ for } i=1,2, \ldots, r \rbrace. 
$$
\end{definition}

\begin{definition}
The polar cone of a cone $K$, denoted by $K^\circ$, is the cone defined by
$$
K^\circ=\lbrace w| w^Tu \leq 0, \mbox{ for all } u \in K \rbrace.
$$
\end{definition}

When we are minimizing a function in a feasible region, we are particularly interested in choosing search directions (descent directions at best) which  improve  the objective function and remain feasible at the same time. Given $x \in \Omega$, we define $K(x, \epsilon)$ as the cone generated by 0 and the outward pointing normals of the constraints within a distance $\epsilon$ of $x$, namely
$$
\lbrace e^{(i)} | u^{(i)}-x^{(i)} \leq \epsilon \rbrace \cup \lbrace -e^{(i)} | x^{(i)}-l^{(i)} \leq \epsilon \rbrace .
$$
In other words, $K(x, \epsilon)$ is generated by the normals to the faces of the feasible region within distance $\epsilon$ of $x$. Observe that if $K(x, \epsilon)=\lbrace 0 \rbrace$, as in the unconstrained case or when $x$ is well within the interior of the feasible region, $K^\circ(x,\epsilon)= \mathbb{R}^{n}$. The cone $K(x, \epsilon)$ is important because, for suitable choices of $\epsilon$, its polar  $K^\circ(x,\epsilon)$ approximates the feasible region near $x$. So, if  $K^\circ(x,\epsilon)\neq \lbrace 0 \rbrace$, the search can proceed from $x$ along all directions in $K^\circ(x,\epsilon)$ for at least a distance of $\epsilon$ and still remain inside the feasible region. See \cite{KLT03} for more details.

\

As in the theory of methods based explicitly on derivatives, in derivative free optimization, we need a measure that lets us know how close a point $x$ is to constrained stationarity. In this article, we adopted the following  measure of stationarity
$$
\chi(x)= \max_{\substack {x +\omega \in \Omega, \\ \| \omega\| \leq 1}} -\nabla f(x)^T \omega. 
$$ 
Roughly speaking, $\chi(x)$ captures the degree to which the direction of steepest descent is outward pointing with respect to the portion of the feasible region near $x$. In \cite{CGT00}, the authors proved that if $\Omega$ is convex, the function $\chi$ has the following properties:
\begin{enumerate}
\item $\chi(x) \mbox{ is continuous. }$
\item $\chi(x)\geq 0$.
\item $\chi(x)=0 \mbox{ if and only if } x \mbox{ is a KKT point }$for the problem (\ref{problem}).
\end{enumerate} 

Thus, showing that $\chi(x_k) \rightarrow 0$ as $k \rightarrow \infty$ establishes a global first-order convergence result, which will be one  of our primary goals on this work.

\ 
\section{The bound constrained nonmonotone pattern search algorithm \texttt{nmps}}

We begin this section by introducing the proposed algorithm.

\

Let $M$ be a positive integer, which indicates how many previous functional values will be considered on the nonmonotone line search. Let $\Delta_{tol}> 0$ be the tolerance for the convergence criterion. Let $D_{\bigoplus}$ be a finite set of $\mathbbm{R}^n$ given by the coordinate directions, that is,
$$
D_{\bigoplus}= \lbrace \pm e^ {{(i)}}| i=1,2,\ldots,n \rbrace.
$$

\

Assume that $\lbrace \eta_k \rbrace$ is a sequence chosen such that $\eta_k>0$, for all $k=0,1,2,\ldots,$ and $ \sum_{k=0}^{\infty}\eta_{k}=\eta < \infty$ is a convergent series. 

\

Suppose that $x_0 \in \Omega$ is an initial approximation to the solution and let $\Delta_0 > 0$ be the initial value for the step length. Given $x_k \in \Omega$, $\Delta_k > \Delta_{tol}> 0$, the steps for computing $x_{k+1}$ are given by the following algorithm.

\begin{algorithm}[\texttt{nmps}] 
\end{algorithm}

\subsubsection*{{\bf Step 1}:} 
Compute $f(x_k)$ and define $f_{\max}(x_k)$ such that
$$
f_{\max}(x_k)=\max \lbrace f(x_k),\ldots,f(x_{k-{\min\lbrace k,M-1\rbrace}}) \rbrace = \max_{0 \leq j\leq m(k)} \lbrace f(x_{k-j}) \rbrace
$$
where $m(k)=\min \lbrace k,M-1 \rbrace$.

\subsubsection*{{\bf Step 2}: {\it Backtracking}}

\begin{itemize}
\item[2.1] Find (if possible) $d\in D_\oplus$ such that $(x_k + \Delta_k d) \in \Omega$ and the inequality
\begin{equation}\label{nmc}
f(x_k + \Delta_k d)\leq f_{\max}(x_k)+ \eta_k - \Delta_k^{2}
\end{equation}
holds. Set $\Delta_{k+1}\leftarrow 1$ and $x_{k+1}\leftarrow x_k+ \Delta_k d$.

\item[2.2] If there is no a direction $d\in D_\oplus$ such that $(x_k + \Delta_k d) \in \Omega$ and (\ref{nmc}) holds, then set 

$\Delta_k \leftarrow \displaystyle \frac{\Delta_k}{2}$ and repeat Step 2.1 until
a new $x_{k+1}$ is found.

If $\Delta_{k+1}<\Delta_{tol}$ terminate the execution of the algorithm.
\end{itemize}

\section{Theoretical results and convergence analysis}

In order to prove our main convergence result we need to demonstrate some auxiliary results. The following proposition concerns the nonmonotonicity of $\{f(x_k)\}$. It is analogous to Lemma 2.3 in \cite{BMR03}.

\begin{proposition} \label{propaux}
If $l(k)$ is an integer such that 
\begin{center}
$kM-m(kM)\leq l(k)\leq kM$, 
\end{center}
with $m(kM)=\min\lbrace kM,M-1\rbrace$, and
$$
\begin{array}{rcl}
f(x_{l(k)})& = & \stackbin[0 \leq j\leq m(kM)]{}{\max} \lbrace f(x_{kM-j}) \rbrace \\[3mm]
& = & \max \lbrace f(x_{kM}),f(x_{kM-1}),\ldots,f(x_{kM-{\min\lbrace kM,M-1\rbrace}}) \rbrace 
\end{array}
$$
then   
\begin{equation}\label{proposicion1}
f(x_{l(k+1)} \leq f(x_{l(k)}) + \eta_{kM} + \ldots + \eta_{kM-M+1} - \Delta_{l(k+1)-1}^{2}
\end{equation}
\end{proposition}

\begin{proof}

From now on we are going to suppose that the iteration $k$ is such that $k \geq M $, therefore $m(kM)=\min \lbrace {kM,M-1} \rbrace= M-1$.

\

Then, by an inductive argument on $t=1,2, \ldots, M$ we will prove that
\begin{equation}\label{induccion}
f(x_{kM+t)} \leq f(x_{l(k)})+ \eta_{kM} + \ldots + \eta_{kM-M+1} - \Delta_{kM+t-1}^{2}
\end{equation}
holds for all iteration $k=1,2, \ldots$

\

In fact, from (\ref{nmc}), we have
$$
\begin{array}{rcl}
f(x_{kM+1}) & \leq & f_{\max}(x_{kM}) + \eta_{kM} - \Delta_{kM}^2 \\ 
& = & \stackbin[0 \leq j\leq M-1]{}{\max} \lbrace f(x_{kM-j}) \rbrace + \eta_{kM} - \Delta_{kM}^2 \\[2mm] 
& = & f(x_{l(k)})+ \eta_{kM} - \Delta_{kM}^2 \\ 
& < & f(x_{l(k)})+ \eta_{kM}, 
\end{array}
$$
for all $k \in \mathbbm{N}$. Therefore, the inequality (\ref{induccion}) holds for $t=1$.

\

Now, by inductive hypothesis, we suppose that 
$$
\begin{array}{rcl}
f(x_{kM+t^\prime)} & \leq & f(x_{l(k)})+ \eta_{kM} + \ldots + \eta_{kM+t^\prime-1} - \Delta_{kM+t^\prime-1}^{2} \\ 
& < & f(x_{l(k)}) + \eta_{kM} + \ldots + \eta_{kM+t^\prime-1}
\end{array}
$$
for all $t^\prime = 1,2,\ldots, t$.
 
\

We will prove that (\ref{induccion}) holds for $t+1$. Indeed,
$$
\begin{array}{rcl}
f(x_{kM+t+1)} & \leq & f_{\max}(x_{kM+t})+ \eta_{kM+t} - \Delta^{2}_{kM+t} \\
& = & \max\lbrace f(x_{kM+t}), f(x_{kM+t-1}), \ldots, f(x_{kM+t-M+1}) \rbrace + \eta_{kM+t} - \Delta^{2}_{kM+t} \\
& = & \max \lbrace \max \lbrace f(x_{kM}), \ldots, f(x_{kM+t-M+1}) \rbrace, f(x_{kM+1}), \ldots, f(x_{kM+t}) \rbrace + \\ 
&& \eta_{kM+t} - \Delta^{2}_{kM+t} \\
& = &  \max \lbrace f(x_{l(k)}), f(x_{kM+1}), \ldots, f(x_{kM+t}) \rbrace + \eta_{kM+t} - \Delta^{2}_{kM+t}
\end{array}
$$

Now by induction step,
$$
\max \lbrace f(x_{kM+1}), \ldots, f(x_{kM+t}) \rbrace < f(x_{l(k)}) + \eta_{kM} + \ldots + \eta_{kM+t+1}.
$$
  
Thus, 

$$
f(x_{kM+t+1)} \leq f(x_{l(k)})+ \eta_{kM} + \ldots + \eta_{kM+t-1} + \eta_{kM+t} -\Delta^{2}_{kM+t},
$$
so the inequality (\ref{induccion}) holds for $t+1$. 

\

Finally, since $(k+1)M-M+1\leq l(k+1) \leq kM+M,$ we have
$$
l(k+1)=kM+t \mbox{ for some } t \in \lbrace 1,2, \ldots, M \rbrace.
$$ 
\end{proof}

The next proposition is an important tool to prove global convergence of the \texttt{nmps} algorithm where  the inequality of Proposition \ref{propaux} is applied iteratively. This idea has also been  introduced by Birgin et. al in \cite{BMR03} and we adapted it for our case.  

\begin{proposition}
If $\lbrace f(x_{k}) \rbrace _{k \in \mathbbm{N}}$ is bounded below then $\displaystyle \lim_{k \to \infty}\Delta^{2}_{l(k)-1}=0$

\end{proposition}

\begin{proof}
By applying the inequality (\ref{proposicion1}) we have 
$$
f(x_{l(k+1)}) \leq f(x_0) + \displaystyle \sum_{k=0}^{\infty}\eta_{k} - \displaystyle \sum_{k=1}^{\infty} \Delta^2_{l(k)-1},
$$
equivalently
$$
\displaystyle \sum_{k=1}^{\infty} \Delta^2_{l(k)-1} \leq f(x_0) -f(x_{l(k+1)}) + \displaystyle \sum_{k=0}^{\infty}\eta_{k}. 
$$
Now, since $f$ is bounded below we have $-f(x_k)\leq -C$ for all $k$, and due to summability of the sequence $\lbrace \eta_k \rbrace$, we obtain
$$
\displaystyle \sum_{k=1}^{\infty} \Delta^2_{l(k)-1}< +\infty,
$$
so $\displaystyle \lim_{k \to \infty}\Delta^{2}_{l(k)-1}=0$, as we want to prove.
\end{proof}

\

In consequence, we observe that 
$$
\lim_{k \to \infty}\Delta_{l(k)-1} = 0,
$$
since the steps $\Delta_k$ are small enough and positives.

\

Now we define set of index
$$
U= \lbrace l(1)-1, l(2)-1, l(3)-1, \ldots \rbrace
$$ 
where $\lbrace l(k) \rbrace$ is the sequence of index defined in the Proposition \ref{propaux}.

\ 

The following two results have been demonstrated by Kolda, Lewis and Torczon in \cite{KLT03}. 
  
\begin{proposition} \label{prop3}
Let $x \in \Omega$ and $\varepsilon \leq 0$, and let $K=K(x,\varepsilon)$ and $K^{\circ}=K^{\circ}(x,\varepsilon)$ for the bound constrained problem (\ref{problem}). Let $\textit{G}_{K^{\circ}}\subseteq D_{\bigoplus}$ the set of generators of $K^{\circ}$. Then, if $[-\nabla f(x)]_{K^{\circ}}\neq 0$, there is $d \in \textit{G}_{K^{\circ}}$ such that 
$$
\frac{1}{\sqrt{n}} \| [-\nabla f(x)]_{K^{\circ}} \| \leq -\nabla f(x)^{T}d.
$$
\end{proposition}

\

\begin{proposition}\label{prop4}
Let $x \in \Omega$ and $\varepsilon \geq 0$, and let $K^{\circ}=K^{\circ}(x,\varepsilon)$ and $K=K(x,\varepsilon)$ for the bound constrained problem (\ref{problem}). Then
$$
\chi (x)\leq \| [-\nabla f(x)]_{K^{\circ}} \| + \sqrt{n} \| [-\nabla f(x)]_{K} 
\|\varepsilon. 
$$
\end{proposition}

\

Next, we present the main global convergence result of \texttt{nmps} algorithm.

\begin{theorem}
Let $f:\mathbbm{R}^n\rightarrow \mathbbm{R}$ be continuously differentiable, and suppose $\nabla f(x)$ is Lipschitz continuous with constante $L$, $\|\nabla f(x) \| \leq \gamma$, for all $x \in \Omega$ and $\lbrace f(x_{k}) \rbrace _{k \in \mathbbm{N}}$ is bounded below. If $\lbrace x_k \rbrace _{k \in U}$ is the sequence generated by the \texttt{nmps} algorithm then
$$
\chi (x_k) \leq \sqrt{n}(L+ \gamma)\Delta_k \, \mbox{ for all } \, k \in U.
$$
\end{theorem}

\begin{proof}
We will consider two cases. 

\textit{Case 1.} If $x_k + \Delta_k d \notin int(\Omega)$ for all $d \in D_{\bigoplus}$, then $x_k + \Delta_k d$ is either on the boundary of or outside of $\Omega$ for all directions $d \in D_{\bigoplus}$.

\

In other words, if $l^{(i)}\leq x_k^{(i)} \leq u^{(i)}$ then $x_k^{(i)}-\Delta_k \leq l^{(i)}$ and $x_k^{(i)}+\Delta_k\geq u^{(i)}$ for all $i = 1, 2,\ldots, n$.

\

The last inequalities imply that if $x_k+\omega \in \Omega$, the vector $\omega$ cannot have their components greater than $\Delta_k$, that is, $\omega^{(i)}\leq \Delta_k$ for all $i$. Therefore, $\|\omega \|\leq \sqrt{n} \Delta_k$.

\

So,
$$
\begin{array}{rcl}
\chi(x_k) & = &\displaystyle \max_{\substack {x_k +\omega \in \Omega, \\ \| \omega\| \leq 1}} -\nabla f(x_k)^T \omega \\
& \leq & \displaystyle \max_{\substack {x_k +\omega \in \Omega, \\ \| \omega\| \leq 1}}\|\nabla f(x_k) \| \| \omega \| \\
& \leq &\|\nabla f(x_k) \| \sqrt{n} \Delta_k \\ 
& \leq & \sqrt{n} \gamma \Delta_k,
\end{array}
$$
which completes the proof for the Case 1.

\

\textit{Case 2.} Now we suppose that there is at least $d \in D_{\bigoplus}$ such that $x_k + \Delta_k d \in int(\Omega)$. Thus, the cone $K^{\circ}(x_k,\Delta_k)$ is generated by all the directions $d \in D_{\bigoplus}$ such that $x_k + \Delta_k d int(\Omega)$.

\

By the mean value theorem, we have that  
\begin{equation}\label{valormedio}
f(x_k+\Delta_k d_k)-f(x_k)=\Delta_k \nabla f(x_k+\lambda_k \Delta_k d_k)^Td_k,
\end{equation}
for some $\lambda_k \in [0,1]$.

\

Since $k \in U$, this implies 
$$
0 \leq f(x_k+ \Delta_k d_k)- f_{\max}(x_k)- \eta_k + \Delta_k^2.
$$

\

Taking into account that $-f_{\max}(x_k) \leq -f(x_k)$ and $\eta_k > 0$ for all $k$, we obtain  
\begin{equation}\label{desigualdad}
0 \leq f(x_k+ \Delta_k d_k) -f(x_k) + \Delta_k^2.
\end{equation}

\

Then, we replace (\ref{valormedio}) in (\ref{desigualdad})
$$
0 \leq \Delta_k \nabla f(x_k+\lambda_k \Delta_k d_k)^Td_k + \Delta_k^2.
$$

\

Next, we divide the last inequality by $\Delta_k$ and adding $-\nabla f(x_k)^Td_k$, we get the following inequality
$$
-\nabla f(x_k)^Td_k \leq (\nabla f(x_k+ \lambda_k \Delta_k d_k)- \nabla f(x_k))^T d_k + \Delta_k.
$$

\

Using the Proposition \ref{prop3} we have
$$
\dfrac{1}{\sqrt{n}} \| [-\nabla f(x)]_{K^{\circ}} \| \leq (\nabla f(x_k+ \lambda_k \Delta_k d_k)- \nabla f(x_k))^T d_k + \Delta_k.
$$
 
Then by the Cauchy-Schwarz inequality, the fact that $\| d_k \|=1$ for all $k$ and the boundedness hypothesis of the gradient, we obtain
$$
\dfrac{1}{\sqrt{n}} \| [-\nabla f(x)]_{K^{\circ}} \| \leq \|(\nabla f(x_k+ \lambda_k \Delta_k d_k)- \nabla f(x_k)) \| + \Delta_k \leq L \Delta_k + \Delta_k.
$$

In consequence,
$$
\| [-\nabla f(x)]_{K^{\circ}} \| \leq \sqrt{n} L \Delta_k + \sqrt{n} \Delta_k \leq \sqrt{n} L \Delta_k.
$$

Finally, combining the Proposition \ref{prop4} with the above result
$$
\chi (x)\leq \| [-\nabla f(x)]_{K^{\circ}} \| + \sqrt{n} \| [-\nabla f(x)]_{K} \|\varepsilon \leq \| [-\nabla f(x)]_{K^{\circ}} \| + \sqrt{n} \gamma \Delta_k \leq \sqrt{n} L \Delta_k + \sqrt{n} \gamma \Delta_k 
$$
consequently,
\begin{center}
$\chi (x)\leq \sqrt{n}(L+ \gamma)\Delta_k$,
\end{center}
and the proof is complete.

\end{proof}

\section{Numerical results}

In this section we show and analyse the numerical results obtained using our nonmonotone pattern search bound constrained optimization \texttt{nmps} algorithm. All the numerical experiments were executed on a computer with a 2.3 GHz Intel Core i5--6200u processor (8 GB RAM). We implemented the  \texttt{nmps} algorithm in \texttt{matlab R2016b 64-bit}. 

\

To the purpose of carefully analyse the performance of our algorithm we decided to organize our study in two parts. First, we compare the \texttt{nmps} algorithm with the  \texttt{patternsearch} routine from \texttt{matlab}'s optimization toolbox, since both algorithms are based on pattern search methods. Then, we study the performance of \texttt{nmps} algorithm using different line search strategies \cite{C09,KLNS15,NS13,ZH04} and the classical Armijo's rule \cite{A66}.

\

We have selected a set of 63 bound constrained problems from Hock--Schittkowski collection \cite{HS81}. Since this collection has only 9 bound constrained problems, we have modified other 54 problems with general constraints, extracting the linear and nonlinear constraints from each one of them. The detailed list of these problems and their characteristics is provided in 
Table~\ref{tab:infoproblemas}.

\

As it is usual in derivative-free optimization articles, we are interested in the number of functional values needed to satisfy the stopping criteria, which are: reaching a sufficiently small step length ($\Delta_k < \Delta_{\mbox{tol}}$), attaining the maximum number of function evaluations $MaxFE$ or attaining the maximum number of iterations $MaxIt$. We adopt the convergence test proposed in \cite{MW09} to measure the ability of an algorithm to improve an initial approximation and to declare that a problem has been solved if the following condition holds

\begin{equation} \label{test}
f(x_0)-f(x) \geq (1-\tau)(f(x_0)-f_{\mbox{L}}),
\end{equation}
where $x_0$ is the initial feasible approximation, $\tau > 0$ is the level of accuracy and $f_{\mbox{L}}$ is the smallest functional value obtained among the considered solvers. We use the performance profile graphs  \cite{DM02,MW09} to illustrate the results obtained with (\ref{test}). 

\

\begin{table}
\begin{center}
\resizebox{8cm}{!} {
\begin{tabular}{|c|c|c|c|c|}
\hline
Prob. & N$^{\circ}$. HS & n & bound & objective \\
 & & & constraints & function \\
\hline
1 & 1 & 2 & 1 & Generalized polynomial \\ \hline
2 & 2 & 2 & 1 & Generalized polynomial\\ \hline
3 & 3 & 2 & 1 & Generalized polynomial\\ \hline
4 & 4 & 2 & 2 & Generalized polynomial\\ \hline
5 & 5 & 2 & 4 & General\\ \hline
6 & 25 & 3 & 6 & Sum of squares\\ \hline
7 & 38 & 4 & 8 & Generalized polynomial\\ \hline
8 & 45 & 5 & 10 & Constant\\ \hline
9 & 110 & 10 & 20 & General\\ \hline
10 & 13 & 2 & 2 & Quadratic\\ \hline
11 & 15 & 2 & 1 & Generalized polynomial\\ \hline
12 & 16 & 2 & 3 & Generalized polynomial\\ \hline
13 & 17 & 2 & 3 & Generalized polynomial\\ \hline
14 & 18 & 2 & 4 & Quadratic\\ \hline
15 & 19 & 2 & 4 & Generalized polynomial\\ \hline
16 & 20 & 2 & 2 &  Generalized polynomial\\ \hline
17 & 21 & 2 & 4 & Quadratic\\ \hline
18 & 23 & 2 & 4 & Quadratic\\ \hline
19 & 24 & 2 & 2 & Generalized polynomial\\ \hline
20 & 30 & 3 & 6 & Quadratic\\ \hline
21 & 31 & 3 & 6 & Quadratic\\ \hline
22 & 32 & 3 & 3 & Quadratic\\ \hline
23 & 33 & 3 & 4 & Generalized polynomial\\ \hline
24 & 34 & 3 & 6 & Linear\\ \hline
25 & 35 & 3 & 3 & Quadratic\\ \hline
26 & 36 & 3 & 6 & Generalized polynomial\\ \hline
27 & 37 & 3 & 6 & Generalized polynomial\\ \hline
28 & 41 & 4 & 8 & Generalized polynomial\\ \hline
29 & 42 & 4 & 2 & Quadratic\\ \hline
30 & 44 & 4 & 4 & Quadratic\\ \hline
31 & 53 & 5 & 10 & Quadratic\\ \hline
32 & 54 & 6 & 12 & General\\ \hline
33 & 55 & 6 & 8 & General\\ \hline
34 & 57 & 2 & 2 & Sum of squares\\ \hline
35 & 59 & 2 & 4 & General\\ \hline
36 & 60 & 3 & 6 & Generalized polynomial\\ \hline
37 & 62 & 3 & 6 & General\\ \hline
38 & 63 & 3 & 3 & Quadratic\\ \hline
39 & 64 & 3 & 3 & Generalized polynomial\\ \hline
40 & 65 & 3 & 6 & Quadratic\\ \hline
41 & 66 & 3 & 6 & Linear\\ \hline
42 & 68 & 4 & 8 & General\\ \hline
43 & 69 & 4 & 8 & General\\ \hline
44 & 71 & 4 & 8 & Generalized polynomial\\ \hline
45 & 72 & 4 & 8 & Linear\\ \hline
46 & 73 & 4 & 4 & Linear\\ \hline
47 & 74 & 4 & 8 & Generalized polynomial\\ \hline
48 & 75 & 4 & 8 & Generalized polynomial\\ \hline
49 & 76 & 4 & 4 & Quadratic\\ \hline
50 & 80 & 5 & 10 & General\\ \hline
51 & 81 & 5 & 10 & General\\ \hline
52 & 83 & 5 & 10 & Quadratic\\ \hline
53 & 84 & 5 & 10 & Quadratic\\ \hline
54 & 86 & 5 & 5 & Generalized polynomial\\ \hline
55 & 93 & 6 & 6 & Generalized polynomial\\ \hline
56 & 101 & 7 & 14 & Generalized polynomial\\ \hline
57 & 102 & 7 & 14 & Generalized polynomial\\ \hline
58 & 103 & 7 & 14 & Generalized polynomial\\ \hline
59 & 104 & 8 & 16 & Generalized polynomial\\ \hline
60 & 106 & 8 & 16 & Linear\\ \hline
61 & 108 & 9 & 2 & Quadratic\\ \hline
62 & 114 & 10 & 20 & Quadratic\\ \hline
63 & 119 & 16  & 32 & Generalized polynomial\\ \hline

\end{tabular}
}
\caption{Characteristics of test problems.}
\label{tab:infoproblemas}
\end{center}
\end{table}

Given $P$ the set of problems, $|P|$ denotes the cardinality of $P$ and $S$ the set of considered solvers. The performance profile of a solver $s \in S$ is defined as the fraction of problems where the performance ratio is at most $\alpha$, that is, $\rho_{s}(\alpha)=\frac{1}{|P|}\mbox \, \mbox{size} \, \lbrace p \in Pv: r_{p,s}\leq \alpha \rbrace$, where $r_{p,s}=\frac{t_{p,s}}{\lbrace  \min t_{p,s}: s \in S\rbrace}$, $t_{p,s}$ is the number of function evaluations required to satisfy the convergence test (\ref{test}).

\

We used the same initial approximation $x_0$ as indicated in \cite{HS81}, projecting onto the bound constraints if the initial approximation was not feasible. After some preliminary tests we adopted $\eta_k=1.1^{-k}$ for all $k$. Finally, the following algorithmic parameters were set: $\Delta_0 = 1.0$, as the initial step length, $M = 15$, $MaxFE = 2500$, $MaxIt = 5000$ and $\Delta_{tol} = TOL = 10^{-6}$.

\

It is worth mentioning three important implementation details of our algorithm. First, the function $f$ is evaluated in all possible coordinate directions and the accepted new approximation is such that produces the minimum functional value of $f(x_k + \Delta_k d)$. Second, the step length is updated using the following scheme $\Delta_{k+1}= \min \lbrace 1, 2\Delta_k \rbrace$. Third, the
accepted points are stored in a memory for the purpose to avoid revisiting older points without slow down the implementation. See \cite{LST07}.

\subsection{Comparison of performance profiles between \texttt{nmps} and \texttt{patternsearch} } \label{experimentos1}

We tested our algorithm \texttt{nmps} using the set of test problems and we compared it with the well established routine \texttt{patternsearch} from \texttt{matlab}. Since both codes are based on a pattern search scheme, we set the same algorithmic parameters. The numerical results are shown in Table~2.

\

In Figure \ref{f:performanceprofile63} we show the performance profile pictures using condition (\ref{test}) with three levels of accuracy: $\tau=10^{-1}$, $\tau=10^{-3}$ and $\tau=10^{-5}$, where a smaller value of $\tau$ means the satusfaction of condition (\ref{test}) is more strict. In a performance profile plot, the top curve represents the most efficient method within a factor $\tau$ of the best measure. When both methods match with the best result, then they are both counted as successful. This means that the sum of the successful percentages may exceed $100\%$.

\begin{figure} [H]
 \centering
  \subfloat[]{
   \label{f:tol1}
    \includegraphics[width=0.5\textwidth]{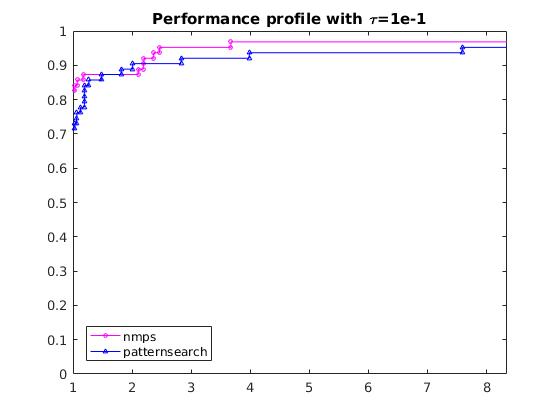}}
  \subfloat[]{
   \label{f:tol2}
    \includegraphics[width=0.5\textwidth]{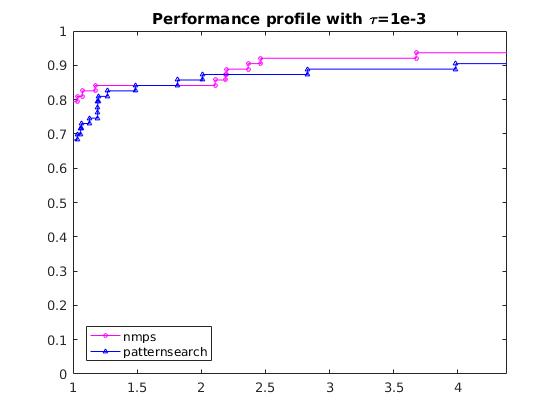}}
    \\    
    \subfloat[]{
    \label{f:tol3}
    \includegraphics[width=0.5\textwidth]{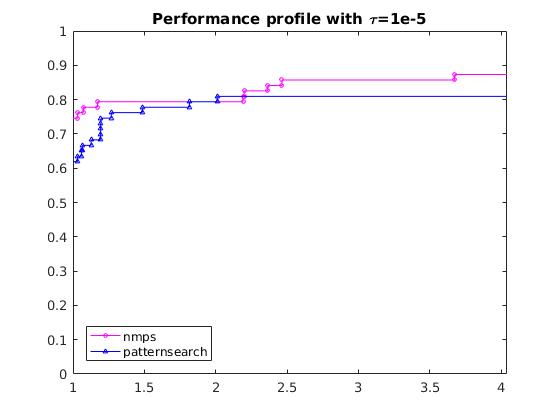}}
 \caption{Performance profiles of \texttt{nmps} and \texttt{patternsearch}.}
 \label{f:performanceprofile63}
\end{figure}

\

In Figure \ref{f:tol1}, with $\tau=10^{-1}$, we observe that \texttt{nmps} algorithm is the best solver in the $82\%$ of the set problems while \texttt{patternsearch} does it in $71\%$. We also see, within a factor of $1.7$ of the best solver, both algorithms have a similar behaviour  and the performance profile shows these algorithms can solve a problem with a probability of $0.87$ with respect to the best solver. Finally, if the goal is to solve efficiently $95\%$ of the problems, the \texttt{nmps} algorithm accomplishes this by using $2.4$ times the minimum number of function evaluations while \texttt{patternsearch} needs a factor of $7.5$.

\

Now, when we increase the level of accuracy to $\tau=10^{-3}$ (Figure \ref{f:tol2}), we note that  \texttt{nmps} wins in the $79\%$ of the problems in comparison with the $68\%$ of \texttt{patternsearch}. Also, both solvers are equivalent if the solution is required within a factor of $1.7$ of the best solver, with a probability of $0.84$. Although both solvers can  reach the solution in, at most, $95\%$ of the problems, our algorithm was closer than the other since it solved $94\%$ of them, using $3.7$ times the minimum number of function evaluations. 

\

Finally, we observe in Figure \ref{f:tol3} that the performance of both solvers deteriorate using 
$\tau=10^{-5}$ as the level of accuracy. In any case, the algorithm \texttt{nmps} wins in $74\%$ of the problems meanwhile \texttt{patternsearch} wins in $62\%$. As before, \texttt{nmps} algorithm 
performs better than \texttt{patternsearch} if you consider a solver that finds the solution using $1.7$ times the minimum number of function evaluations, with a probability of $0.8$. In such a case, we could expect at most that \texttt{nmps} algorithm solves $87\%$ of the problems  while \texttt{patternsearch}  solves $81\%$.

\

We conclude that, regardless the level of accuracy, our algorithm outperforms the  \texttt{patternsearch} routine with the set of test problems considered. In fact, our algorithm always has a probability $10\%$ greater than \texttt{patternsearch} to get the solution.
 
\

In the next subsection we will test our algorithm using other line search strategies in order to analyse and understand the advantages of employing different nonmonotone line search procedures.

\

\subsection{Comparison of performance profiles using other line search strategies} \label{otrasbusquedas}

Recently, some authors proposed different nonmonotone line search strategies for solving unconstrained minimization problems and nonlinear systems, with derivatives based and derivative-free methods \cite{C09,KLNS15,NS13,ZH04}. They show that such approaches are successful, therefore we have adopted them to our bound constrained derivative-free optimization problem (\ref{problem}). 

\

The first approach, which is called {\it C-line search} and we have implemented in \texttt{Cpatternsearch} algorithm, is similar to the nonmonotone line search condition (\ref{nmc}), where $f_{\max}(x_k)$ is replaced by the sequence $\{C_k\}$ given by 
\begin{equation} \label{C}
Q_{k+1}=r_k Q_k +1, \hspace{1cm} C_{k+1}=\dfrac{r_kQ_k(C_k+ \eta_k)+f_{k+1}}{Q_{k+1}}
\end{equation}
with $Q_0=1$, $C_0=f(x_0)$, $r_k \in [0,1]$ and $\lbrace \eta_k \rbrace$ satisfying	 $\sum_{k=0}^{\infty}\eta_{k}=\eta < \infty$ for all $k=0,1,2,\ldots$

\

The second strategy, which is called {\it $\lambda$-line search} and we have implemented in \texttt{$\lambda$patternsearch} algorithm, is also analogous to (\ref{nmc}) but in this case $f_{\max}(x_k)$ is defined by
\begin{equation} \label{lambda}
f_{\max}(x_k)=\max \lbrace f(x_k),\displaystyle \sum_{r=0}^{m(k)-1} \lambda_{k_r}f(x_{k-r}) \rbrace
\end{equation}
with $M \in \mathbbm{N}$, $m(k)=\min \lbrace k,M-1 \rbrace$, $\lambda_{k_r}\geqslant \lambda$ and $\sum_{r=0}^{m(k)-1}\lambda_{k_r}=1$ for all $k=0,1,2,\ldots$

\

Finally, we have adapted the classical Armijo's rule to our bound constrained problem (\ref{problem})
\begin{equation}\label{armijo}
f(x_k + \Delta_k d)\leq f(x_k)+ \eta_k - \Delta_k^{2}
\end{equation}
with $\lbrace \eta_k \rbrace$ y $\Delta_k$ defined as in the Algorithm 1. It was implemented in \texttt{armijo} algorithm.

\

Again, we chose $\eta_k=1.1^{-k}$ for all $k$, for the three new conditions. Also, we adopted $M=15$ and $\lambda_{k_r}=1/m(k)$ for all $r$ in \texttt{$\lambda$patternsearch} and $r_k=0.85$ for all $k$ in \texttt{Cpatternsearch}. Next, we show the performance profiles with convergence test (\ref{test}) for solving our set of test problems using \texttt{nmps}, \texttt{patternsearch}, \texttt{Cpatternsearch}, \texttt{$\lambda$patternsearch} and \texttt{armijo} algorithms. The numerical results are also presented in Table~2.

\

In Figure \ref{f:tol11}, for $\tau=10^{-1}$, we see that \texttt{$\lambda$patternsearch} 
attains the best performance in $76\%$ of the problems, followed by \texttt{nmps} with $59\%$ and \texttt{Cpatternsearch} with  $57\%$, below  we find \texttt{patternsearch} and \texttt{armijo} with $56\%$ and $51\%$ respectively. We also observe, within a factor of $2.5$ of the best solver, \texttt{$\lambda$patternsearch} reaches the greater probability of solving a problem (around $0.92$) and \texttt{nmps} follows it with $0.9$. Furthermore, in this case, \texttt{patternsearch} is the solver with the lowest performance, with a probability of $0.76$. Moreover, with this level of accuracy, we note that \texttt{armijo} solves almost $97\%$ of the problems using a factor of $4.9$ times the minimum number of function evaluations, meanwhile \texttt{nmps} has the same behaviour requiring a factor of $5.75$.

\

Figure \ref{f:tol22} displays the performance profile for $\tau=10^{-3}$, increasing in this way the level of accuracy. Once more, \texttt{$\lambda$patternsearch} is the most successful solver with a probability of $0.63$, followed by \texttt{Cpatternsearch} and \texttt{nmps} with a probability of 
$0.62$ and $0.6$, respectively. In the last positions, we find \texttt{armijo} and \texttt{patternsearch} with $0.52$ and $0.51$, respectively. Now, within a factor of $2.5$ of the best solver, \texttt{nmps} exhibits the best performance solving  $89\%$ of the problems. Later, 
\texttt{Cpatternsearch} and \texttt{armijo} get $87\% $ and $85\%$ on the resolution of our set of problems. With this level of accuracy, the most we can expect is to solve $94\%$ of the problems within a factor of $3.8$ of the best solver. In the same sense, \texttt{Cpatternsearch} and \texttt{nmps} present a similar behaviour.

\

Finally, Figure \ref{f:tol33}  presents the performance profile for $\tau=10^{-5}$, requiring in this way a greater descent in the objective function. We observe that \texttt{Cpatternsearch}  wins in $65\%$ of the problems, followed by \texttt{nmps} and \texttt{$\lambda$patternsearch}
with $57\%$, \texttt{armijo} with  $49\%$ and \texttt{patternsearch} with $44\%$. We see, within a factor of $2.5$ of the best solver, \texttt{Cpatternsearch}  obtains the higher probability for solving a problem ($0.87$), followed by \texttt{nmps} ($0.84$). In this figure, as opposed to Figures \ref{f:tol11} and \ref{f:tol22}, we can see how the performance of the different methods are slightly separated from each other. Also, \texttt{Cpatternsearch} algorithm is always on top of all the remaining solvers. In this case, \texttt{Cpatternsearch} achieves the solution in $92\%$  of the problems within a factor of $2.79$ followed by  \texttt{nmps} that solves $87\%$ of the problems using $3.6$ times the minimum number of function evaluations. The other solvers can only solve at most $80\%$ of the problems employing greater factors.

\begin{table}
\begin{center}
\resizebox{16cm}{!} {
\begin{tabular}{|c|c|c|c|c|c|c|c|c|c|c|c|c|c|c|c|}
\hline
Prob. & \multicolumn{3}{c|}{\texttt{nmps}}& \multicolumn{3}{c|}{\texttt{patternsearch}} &\multicolumn{3}{c|}{\texttt{$\lambda$patternsearch}}&
\multicolumn{3}{c|}{\texttt{Cpatternsearch}}&\multicolumn{3}{c|}{\texttt{armijo}}\\
\cline{2-16}

& $It.$ & $FE.$ & $stop$
& $It.$ & $FE.$ & $stop$
& $It.$ & $FE.$ & $stop$
& $It.$ & $FE.$ & $stop$
& $It.$ & $FE.$ & $stop$\\
\hline

1  &  341 & 352  & TOL & 625 & 2500 & MaxFE& 339& 352 &TOL &  530& 352 & TOL& 1003 & 2500 & MaxFE\\ \hline
2  & 402 &  323 & TOL & 58 & 207 & TOL &5000 & 6&MaxIt &628 &281&TOL & 277 & 155 & TOL\\ \hline
3  & 277 & 83 & TOL & 39 &  119 & TOL &	1660& 2500& MaxFE&307 &83& TOL&1659 & 2500 &MaxFE\\ \hline
4  & 277 & 46 & TOL &34 & 84 & TOL &5000 & 14 &MaxIt &307 &46& TOL& 277 & 48& TOL\\ \hline
5  & 293 & 305 & TOL & 36 & 145 & TOL  &269 & 337&TOL & 305&314& TOL& 269& 336&TOL\\ \hline
6  & 700 & 2500 & MaxFE & 513 & 2500 & MaxFE &5000 & 1330&MaxIt &670 &2500& MaxFE& 552& 2500& MaxFE\\ \hline
7  & 1124 & 2500 & MaxFE & 176 & 1409 & TOL &5000 & 177&MaxIt &974 &2500& MaxFE& 455& 2500& MaXFE\\ \hline
8  & 278 & 219 & TOL & 50 & 290 & TOL &	5000& 134& MaxIt& 307&132& TOL& 278 & 248& TOL\\ \hline
9  & 262 & 2500 & MaxFE &129  & 2500 & TOL & 141& 2500& MaxFE&211 &2500&MaxFE & 141& 2500&MaxFE\\ \hline
10 & 278 &  64 & TOL & 26 & 77 & TOL &278 & 64&TOL &307 &64& TOL& 278 & 64 &TOL\\ \hline
11  & 277 & 158 & TOL & 625 & 2500 & MaxFE&5000 &106 &MaxIt &306 &156&TOL &277  &539 &TOL\\ \hline
12  & 278 & 150 & TOL & 93 & 343 & TOL & 5000&98&MaxIt &306 &139& TOL& 278& 143&TOL\\ \hline
13  & 278 & 150 & TOL & 93 & 343 & TOL & 5000&98&MaxIt & 306&139& TOL& 278& 143&TOL\\ \hline
14  & 278 & 45 & TOL & 37 & 92  & TOL &5000 &17 &MaxIt & 307&45&TOL & 278& 45&TOL\\ \hline
15  & 278 & 77 & TOL & 48 & 133 & TOL & 278&77&TOL & 5000&77&MaxIt & 278& 77&TOL\\ \hline
16  & 278 & 151 & TOL & 93 & 345 & TOL &5000 &99&MaxIt & 306&140& TOL& 278& 144&TOL\\ \hline
17  & 277 & 63 & TOL & 21 & 64 & TOL & 277&63&TOL & 307&63&TOL & 277 & 63&TOL\\ \hline
18  & 288 & 260 & TOL & 34 & 137 & TOL & 277&289&TOL & 305&314& TOL& 270 & 369& TOL\\ \hline
19  & 834 & 2500 & MaxFE & 627 & 2500 & MaxFE &834 &2500&MaxFE & 834&2500& MaxFE& 834& 2500&MaxFE\\ \hline
20  & 278 & 108 & TOL & 26 & 131 & TOL &5000 &23&MaxIt &307 &108& TOL& 278& 108&TOL\\ \hline
21  & 278 & 106 & TOL & 26 & 128 & TOL & 5000&21&MaxIt & 307&106&TOL & 278& 106&TOL\\ \hline
22  & 277 & 76 & TOL & 47 & 183 & TOL &277 &76& TOL& 306&76& TOL&277 & 82&TOL\\ \hline
23  & 277 & 712 & TOL & 37 & 129 & TOL & 277&708&TOL & 307&791&TOL & 277& 704&TOL\\ \hline
24  & 278 & 480 & TOL & 128  & 747 & TOL &278 &480& TOL& 307&480& TOL& 278& 480&TOL\\ \hline
25  &  289 & 480 & TOL & 130 & 774 & TOL &278 &508&TOL &305 &522& TOL&273 & 625&TOL\\ \hline
26  & 281 & 192 & TOL & 95 & 402 & TOL & 281&192&TOL &5000 &192& MaxIt& 281& 192&TOL\\ \hline
27  & 5000 & 443 & MaxIt & 34 & 86 & TOL & 5000 &443&TOL &5000 &437& MaxIt&5000 & 443&MaxIt\\ \hline
28  & 278 & 81 & TOL & 20 & 81 & TOL & 5000&13&MaxIt &307 &81& TOL&278 & 81&TOL\\ \hline
29  & 278 & 146 & TOL & 43 & 259 & TOL &5000 &44&MaxIt & 307&146&TOL & 278& 146&TOL\\ \hline
30  & 625 & 2500 & MaxFE & 418 & 2500 & MaxFE &625 &2500&MaxFe & 625&2500& MaxFE& 625& 2500&MaxFE\\ \hline
31  & 288 & 800 & TOL & 39 & 391 & TOL &277 &939&TOL & 306&965& TOL& 270& 1130&TOL\\ \hline
32  & 278 & 13 & TOL & 247 & 2500 & MaxFE &278 &13& TOL& 307&287&TOL & 278 & 13& TOL\\ \hline
33  & 277 & 1818 & TOL & 51 & 398 & TOL & 384&2500&MaxFE  & 306&2011& TOL&277 & 1810&TOL\\ \hline
34  & 280 & 140 & TOL & 61 & 191 & TOL &5000 &77&MaxIt & 307&120& TOL&278 & 134&TOL\\ \hline
35  & 278 & 326 & TOL & 119 & 410 &TOL &277 &320&TOL & 307&335& TOL& 277& 509&TOL\\ \hline
36  & 403 & 737 & TOL &  125 & 751 & TOL & 294&746& TOL& 322&755&TOL & 278& 744&TOL\\ \hline
37  & 285 & 287 & TOL & 48 & 210 & TOL &284 &122&TOL & 5000&275& TOL&284 & 122&TOL\\ \hline
38  & 500 & 2500 & MaxFE & 417  & 2500 & MaxFE &500 &2500&MaxFE &500 &2500&MaxFE & 500& 2500&MaxFE\\ \hline
39  & 392 & 1501 & TOL & 325 & 1945 & TOL &5000 &1135&MaxIt & 5000&1447& MaxIt& 634& 2500&TOL\\ \hline
40  & 278 & 106 & TOL & 70 & 297 & TOL &5000 &61&MaxIt & 306&106&TOL & 278& 106&TOL\\ \hline
41  & 278 & 592 & TOL &  140 & 782 & TOL &278 &750&TOL & 307&592&TOL &420 & 1842&TOL\\ \hline
42  & 277 & 1241 & TOL & 46 & 288 & TOL & 277&1234&TOL & 307&1377&TOL & 278& 1238&TOL\\ \hline
43  & 282 & 1373 & TOL & 48 & 299 & TOL &278 &1239&TOL &445 &2069&TOL & 278& 1226&TOL\\ \hline
44  & 278 & 110 & TOL & 53 & 252 & TOL & 5000&38&MaxIt & 307&110&TOL & 278& 110&TOL\\ \hline
45  & 278  & 107 & TOL & 76 & 431 & TOL & 5000&39&MaxIt & 306&107& TOL& 278& 107&TOL\\ \hline
46  & 278 & 100 & TOL & 66 & 352 & TOL & 278&100&TOL & 306&100&TOL & 278& 100&TOL\\ \hline
47  & 278 & 121 & TOL & 20 & 117 & TOL & 278&121&TOL & 307&107&TOL & 278& 121&TOL\\ \hline
48  & 278 & 106 & TOL & 20 & 113& TOL&278 &106& TOL& 307&104&TOL & 278& 106&TOL\\ \hline
49  & 278 & 195 & TOL & 60 & 426& TOL & 278&195&TOL & 306&195&TOL & 277& 177&TOL\\ \hline
50  & 278 & 457 & TOL & 74 & 484&TOL &278 &457&TOL & 307&181& TOL& 278& 457&TOL\\ \hline
51  & 287 & 538 & TOL & 46 & 396& TOL&364 &2500&MaxFE & 318&663&TOL &361 & 2500&MaxFE\\ \hline
52  & 284 & 1289 & TOL & 20 & 101& TOL& 284&1289&TOL & 5000&155&TOL & 284& 1289&TOL\\ \hline
53  & 405 & 2500 & MaxFE & 625 & 2500&MaxFE &405 &2500&MaxFE & 405&2500&MaxFE & 405& 2500&MaxFE\\ \hline
54  & 602 & 1443 & TOL & 85 & 826& TOL&249 &1016&TOL & 450&1275&TOL & 252& 1045&TOL\\ \hline
55  & 277 & 1856 & TOL & 106 & 824& TOL&278 &1861&TOL & 307&2055&TOL & 277& 1849&TOL\\ \hline
56  & 357 & 2500 & MaxFE & 233 & 2500&MaxFE &5000 &547&MaxIt & 361&2500&MaxFE & 247& 2500&MaxFE\\ \hline
57  & 376 & 2500 & MaxFE & 233 & 2500& MaxFE&5000 &547&MaxIt &347 &2500& MaxFE& 255& 2500&MaxFE\\ \hline
58  & 376 & 2500 & MaxFE  & 232 & 2500& MaxFE&5000 &964&MaxIt &345 &2500& MaxFE& 257& 2500&MaxFE\\ \hline
59  & 278 & 621 & TOL & 100 & 1294& TOL & 213&2500& MaxFE& 306&621& TOL& 215& 2500&MaxFE\\ \hline
60  & 167 & 2500 & MaxFE & 157 & 2500& MaxFE&167 &2500&MaxFE & 167&2500&MaxFE & 167& 2500&MaxFE\\ \hline
61  & 167 & 2500 & MaxFE & 157 & 2500& MaxFE&167 &2500& MaxFE&167 &2500& MaxFE& 167& 2500&MaxFE\\ \hline
62  & 149 & 2500 & MaxFE & 150 & 2500& MaxFE&149 &2500& MaxFE& 149&2500&MaxFE & 149& 2500&MaxFE\\ \hline
63  & 278 & 2446 & TOL & 104 & 2500& MaxFE& 5000&2142& MaxIt& 307&2446&TOL & 278& 2446&TOL\\ \hline

\end{tabular}
}
\caption{Results of numerical experiments}
\end{center}
\end{table}

\

\begin{figure}[H]
 \centering
  \subfloat[]{
   \label{f:tol11}
    \includegraphics[width=0.5\textwidth]{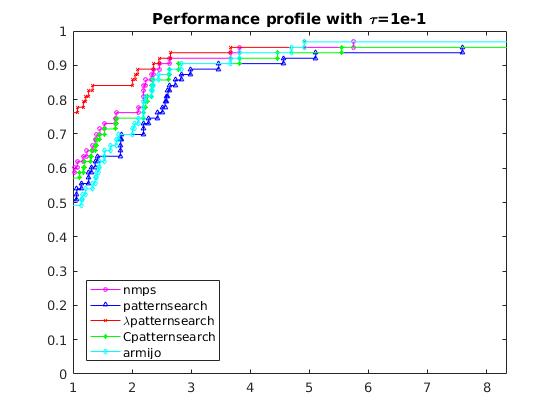}}
  \subfloat[]{
   \label{f:tol22}
    \includegraphics[width=0.5\textwidth]{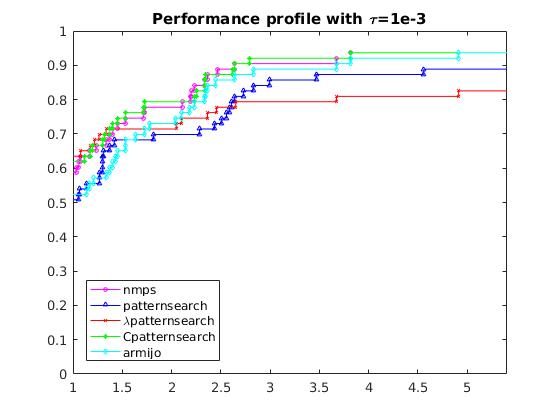}}
    \\    
    \subfloat[]{
    \label{f:tol33}
    \includegraphics[width=0.5\textwidth]{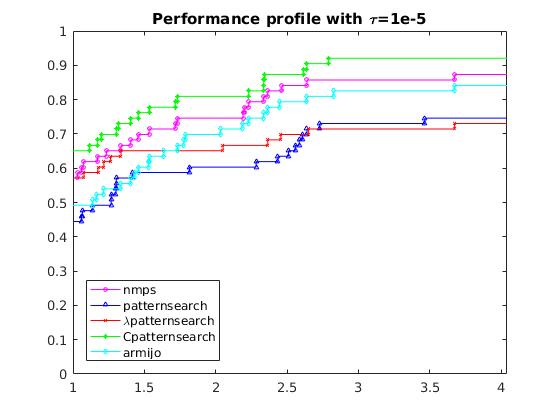}}
\caption{Performance profiles for the five strategies}

\label{f:performanceprofile63_2}
\end{figure}

From the analysis of the figures included in Figure \ref{f:performanceprofile63_2} we can obtain the following conclusions. First, in most cases the use of  a nonmonotone line search of the kind (\ref{nmc}), (\ref{C}) or (\ref{lambda}) turns into an advantage in the performance of the algorithms compared to the classical Armijo's rule (\ref{armijo}). In other words, a greater effort devoted to building $f_{\max}(x_k)$ or $C_k$ results in a decrease in the number of function evaluations carried out by the algorithm, which is one of the main goals in derivative-free methods. Second, although the strategies (\ref{nmc}) and (\ref{lambda}) have similar definitions, the above analysis shows, while  \texttt{$\lambda$patternsearch} reduces its performance as the level of accuracy increases, \texttt{nmps} remains stable for all values of  $\tau$. So, if we should choose between this two strategies, \texttt{nmps} would be  the most suitable. Third, the \texttt{Cpatternsearch} algorithm could be considered as the solver with the best performance because it is always above the other solvers as the level of accuracy increases. Our algorithm \texttt{nmps}, in its turn, seems to be a good competitor to \texttt{Cpatternsearch}, attaining a similar performance with respect to the latter in several cases. Also, \texttt{nmps} always obtains the second place regarding the probability to solve all the problems. At the end, to our surprise, we observe that the performance of \texttt{patternsearch} algorithm is in many cases below all remaining methods.

\

\section{Conclusions}

In this paper, we have proposed a new pattern search algorithm \texttt{nmps} to solve bound constrained optimization problems, which uses a nonmonotone line search strategy for accepting the new iterate. We have proved that, under mild assumptions, we can guarantee global convergence of our method to a KKT point. This result is strongly based on  the relationship between the step length and the stationarity measure (defined by Conn et al. in \cite{CGT00}), and it  was proved in Section~4.  Furthermore, we have performed several numerical experiments where we have compared the performance of our algorithm to other line search strategies that were implemented in  \texttt{patternsearch}, \texttt{$\lambda$patternsearch}, \texttt{Cpatternsearch} and \texttt{armijo} algorithms. The benchmark results were satisfactory, so we can conclude that the \texttt{nmps} algorithm is competitive, compared to the other solvers, as the numerical experiments reveal. We are currently working on an extension of the \texttt{nmps} algorithm to linearly constrained optimization problems, taking into account the directions generated by the linear constraints.

\

\section*{Disclosure statement}

No potential conflict of interest was reported by the authors.

\

\section*{Funding}

This work was partially supported by research project Consejo Nacional de Investigaciones Cient\'ificas y T\'ecnicas (CONICET) PIP 112--201101--00050, PUE 22920160100134CO and Secretar\'ia de Ciencia y Tecnolog\'ia (SECyT-UNC).

\
\bibliographystyle{plain}

\end{document}